\documentclass[11pt]{amsart}

\usepackage{lscape}

\newcounter{cnt1}
\newcounter{cnt2}
\newcommand{\blr}{\begin{list}{$(\roman{cnt1})$}
    {\usecounter{cnt1} \setlength{\topsep}{0pt}
        \setlength{\itemsep}{0pt}}}
\newcommand{\bla}{\begin{list}{$($\alph{cnt2}$)$}
    {\usecounter{cnt2} \setlength{\topsep}{0pt}
        \setlength{\itemsep}{0pt}}}
\newcommand{\el}{\end{list}}
\newtheorem{thm}{Theorem}

\newtheorem{cor}[thm]{Corollary}
\newtheorem{ex}[thm]{Example}

{\newtheorem{Def}[thm]{Definition}
\newtheorem{prop}[thm]{Proposition}
\newtheorem{rem}[thm]{Remark}
\newcommand{\Rem}{\begin{rem} \rm}
\newcommand{\bdfn}{\begin{Def} \rm}
\newcommand{\edfn}{\end{Def}}

\begin{document}
\large
\title[Computing subdifferential limits  of operators on Banach spaces]{ Subdifferential sets of operators }
\author[Rao]{T. S. S. R. K. Rao}
\address[T. S. S. R. K. Rao]
{Department of Mathematics\\
Shiv Nadar University \\
Delhi (NCR) \\ India,
\textit{E-mail~:}
\textit{srin@fulbrightmail.org}}
  \maketitle
\begin{abstract}
Let $X,Y$ be complex Banach spaces. Let ${\mathcal L}(X,Y)$ be the space of bounded operators . An important aspect of understanding differentiability is to study the subdifferential of the norm at a point, say $x \in X$, this is the set, $\{f \in X^\ast:\|f\|=1~,f(x)=\|x\|\}$. See page 7 in \cite{DGZ}. Motivated by recent results of Singla (\cite{S}) in the context of Hilbert spaces, for $T \in {\mathcal L}(X,Y)$, we determine the subdifferential of the operator norm at $T$, $\partial_T = \{\Lambda \in {\mathcal L}(X,Y)^\ast: \Lambda(T) = \|T\|~,~\|\Lambda\|=1\}$. Our approach is based on the `position' of the space of compact operators and the Calkin norm of $T$. Our ideas give a unified approach and extend several results from \cite{S} to the case of $\ell^p$-spaces for $1<p<\infty$. We also investigate the converse, using  the structure of the subdifferential set to decide when the Calkin norm is a strict contraction. As an application of these ideas, we partially solve the open problem of relating the subdifferential of the operator norm at a compact operator $T$ to that of $T(x_0)$, where $x_0$ is a unit vector where $T$ attains its norm.

\end{abstract}
\section { Introduction}
Let $X,Y$ be complex Banach spaces. An interesting problem in operator theory is to determine standard Banach space theoretic sets in ${\mathcal L}(X,Y)$, based on geometric properties of $X,Y$ and their duals. 
\vskip 1em
 Let $X_1$ denote the closed unit ball and let $\partial_e X^\ast_1$ be the set of extreme points of the dual unit ball. {\em In the context of dual spaces, closure operation is always with respect to the weak$^\ast$-topology}. 
For notational simplicity, we call the subdifferential of the norm at a vector $x$, $\partial_x =\{f \in X^\ast: \|f\|=1,~f(x)=\|x\|\}$,  as the subdifferential set, with the vector under consideration appearing in the subscript . This notation is consistent with \cite{S}. Study of these sets in various categories of Banach spaces has long history (an extensive bibliography is available in \cite{S1}) and we only quote the results of Taylor and Werner (see \cite{TW} and \cite{W}) to illustrate.
\vskip 1em

 For $y \in X$, the relation to the directional derivative of the norm at $x$ is given by the formula (see \cite{H} and \cite{DGZ}, pages 7,8)
$$lim_{t\rightarrow 0^+} \frac{\|x+ty\|-\|x\|}{t} = sup \{re(f(y)): f \in \partial_x\}.$$

We recall that $\partial_x$ is an extreme, weak$^\ast$-closed convex subset of $X^\ast_1$ ( called a face, see \cite{A} for basic convexity theory) and therefore the set of extreme points,
 $\partial_e (\partial_x)= \partial_x \cap \partial_e X^\ast_1$. Thus we can replace the functionals in the formula by extreme points. The idea now is to replace the functionals in the RHS set by those coming from an appropriate  subspace of $X$, where the extreme functionals have specific description. Note that the limit on LHS is independent of this manipulation, which is the prime motivation of this approach.
\vskip 1em
For $x^{\ast\ast} \in X^{\ast\ast}$, $y^\ast \in Y^\ast$, $x^{\ast\ast} \otimes y^\ast$ denotes the functional $(x^{\ast\ast} \otimes y^\ast)(T) = x^{\ast\ast}(T^\ast(y^\ast))$. We recall from \cite{HWW} (Theorem VI.1.3) that
for any subspace $J$ of compact operators ${\mathcal K}(X,Y)$ , $\partial_eJ^\ast_1 \subset \{x^{\ast\ast} \otimes y^\ast: x^{\ast\ast} \in \partial_e X^{\ast\ast}_1,~y^\ast \in \partial_e Y^\ast_1\}$ and equality holds if $J$ also contains all finite rank operators. For $A \in {\mathcal L}(X,Y)$, 
such that $d(A, {\mathcal K}(X,Y))<1$, our main result is to determine $\partial_A$ depending on the position of the space of compact operators.
\vskip 1em
We recall from Chapter I of \cite{HWW}, that a closed subspace $J \subset X$ is said to be a $M$-ideal, if there is a linear projection $P: X^\ast \rightarrow X^\ast$ such that $ker(P) = J^\bot$ and $\|x^\ast\|= \|P(x^\ast)\|+\|x^\ast-P(x^\ast)\|$ for all $x^\ast \in X^\ast$. When this happens, functionals in $J^\ast$ have unique norm-preserving extension in $X^\ast$ and $J^\ast$ is canonically identified with the range of $P$ so that $X^\ast = J^\bot \bigoplus_1 J^\ast$, consequently $\partial_e X^\ast_1 = \partial_e J^\ast_1 \cup \partial_e J^\bot_1$.

\vskip 2em
See \cite{HWW} for examples of $M$-ideals from
function spaces, Banach algebras and spaces of operators. In particular Chapter VI deals with examples where ${\mathcal K}(X,Y)$ is a $M$-ideal in ${\mathcal L}(X,Y)$, which includes, ${\mathcal K}(\ell^p)$ in ${\mathcal L}(\ell^p)$ for $1<p<\infty$ and ${\mathcal K}(X,c_0)$ is a $M$-ideal in ${\mathcal L}(X,c_0)$ as well as $\bigoplus_{\infty} X$, for any Banach space $X$. Note also $M$-ideals in $C^\ast$-algebras are precisely closed two-sided ideals. This explains the overlap with \cite{S} and gives possible extensions of the results there beyond $C^\ast$-algebras. See also \cite{S1}.
\vskip 1em
We investigate conditions on $X,Y$ which will allow us to determine the Calkin norm (i.e., the quotient norm of ${\mathcal L}(X,Y)/{\mathcal K}(X,Y)$) of an operator, knowing the behaviour of the subdifferential set. We show that under the general conditions assumed here, if $0$ is not a weak$^\ast$ accumulation point of $\partial_e X^{\ast\ast}_1~,~\partial_e Y^\ast_1$, then the Calkin norm indeed gets determined.
\vskip 1em
Since the subdifferential set $\partial_x$ can be used to evaluate the limit, for $y \in X$, $lim_{t\rightarrow 0^+} \frac{\|x+ty\|-\|x\|}{t} $, a natural question is, when $X$ is canonically embedded in its bidual $X^{\ast\ast}$, if one can replace $y$ in the above formula, by $\tau \in X^{\ast\ast}$, while retaining $\partial_x$
in the RHS?  We partially answer this, using the notion of point of norm-weak upper semi-continuity of the preduality map on $X^{\ast\ast}$ (see the preamble before Proposition 11).
\vskip 1em

We also show that for the special classes of compact operators ${\mathcal K}(X,Y)$, when they form a $M$-ideal in ${\mathcal L}(X,Y)$ and  ${\mathcal L}(X,Y)$ is the  bidual of ${\mathcal K}(X,Y)$ (for example, when $X$ and $Y$ are $\ell^p$-spaces, $1 < p <\infty$), the subdifferential continues to behave well in all higher even order duals of ${\mathcal K}(X,Y)$.
\vskip 1em
Another  interesting  problem is to relate the subdifferential of the operator norm at $T$ to that of $T(x_0)$, where $x_0$ is a point where $T$ attains its norm. Using some recent work from \cite{R1}, we give conditions when this can be solved for a $T \in {\mathcal K}(X,Y)$, for a reflexive space $X$,  $J \subset Y$ is a $M$-ideal.
\vskip 1em
For a compact Hausdorff space $\Omega$, let $C(\Omega,Y)$ be the space of $Y$-valued continuous functions, equipped with the supremum norm. For $\omega \in \Omega$, $y^\ast \in Y^\ast$, $\delta(\omega)\otimes y^\ast$ denotes the functional $(\delta(\omega) \otimes y^\ast)(f)=y^\ast(f(w))$. For a $M$-ideal $J \subset Y$ and $f \in C(\Omega,Y)$, if $\partial_f =\overline{CO}\{\delta(\omega)\otimes y^\ast:y^\ast(f(\omega))=\|f\|,~\omega \in \Omega~,y^\ast\in \partial_e J^\ast_1\}$, (here and elsewhere, $CO$ stands for the convex hull of the set in $\{.\}$), when $0$ is not a weak$^\ast$-closed accumulation point of $\partial_e Y^\ast_1$, there is a $\omega_0 \in \Omega$ such that $\|f(\omega_0)\|=\|f\|$ and $\partial_{f(\omega_0)}=\overline{CO}\{y^\ast \in \partial_e J^\ast_1: y^\ast(f(\omega_0))= \|f\|\}$.
\section{Geometry of the subdifferential set}
We first give a general form of Theorem 1.4 from \cite{S} (see also the Zbl review by this author). Let $J \subset X$ be a closed subspace. We note that if $x \in J$, then relative to $J$, $\partial_x = \overline{CO} \{f \in \partial_e J^\ast_1: f(x)=1\}$, where the closure is in the weak$^\ast$-topology of $J^\ast$.

\vskip 1em
\begin{thm}
Let $J \subset X$ be a $M$-ideal. Let $x \in X$ be a unit vector such that $d(x,J)<1$. Then
$\partial_x = \overline{CO}\{f \in \partial_e J^\ast_1: f(x) = 1\}$. Here the closure is taken in the weak$^\ast$-topology of $X^\ast$. In particular $\partial_x$ has the same description as when $x \in J$.
Thus  for $y \in X$, $$lim_{t\rightarrow 0^+} \frac{\|x+ty\|-\|x\|}{t} = sup \{re(f(y)): f \in \partial_e J^\ast_1,~f(x)=1\}.$$

\end{thm}
\vskip 2em
\begin{proof}
	
Since $J$ is a $M$-ideal, we have, $X^\ast = J^\ast \bigoplus_1 J^\bot$. Thus if the set $\{f \in \partial_eJ^\ast_1: f(x) = 1\}$ is non-empty, then it is contained in $\partial_x$.
\vskip 1em
Since $\partial_x$ is a weak$^\ast$-closed extreme convex subset of $X^\ast_1$, we will show that any extreme point $g \in \partial_x$ is in $J^\ast$ . This completes the proof.
\vskip 1em
Since $g \in \partial_e X^\ast_1$, in view of decomposition of $X^\ast$, suppose $g \in J^\bot$. Now $g(x) = 1$ implies $d(x,J) = 1$.
Therefore, $g \in \partial_eJ^\ast_1$.
\end{proof}
\begin{rem}
Since $y$ is  weak$^\ast$-continuous on $J^\ast$, when $\partial_e J^\ast_1$ is weak$^\ast$-closed in the weak$^\ast$-topology of $J^\ast$, we have
$$lim_{t\rightarrow 0^+} \frac{\|x+ty\|-\|x\|}{t} = re(f_0(y))$$ for some $ f_0 \in \partial_e J^\ast_1$, $f_0(x)=1$. This in particular always happens (under the above hypothesis) when $\partial_e X^\ast_1$ is weak$^\ast$-closed.
	
\end{rem}

\vskip 1em
The following corollary follows immediately from our specific knowledge of extreme points of ${\mathcal K}(X,Y)^\ast_1$. 
\begin{cor}
	Let $X,Y$ be Banach spaces and let $J \subset {\mathcal K}(X,Y) \subset {\mathcal L}(X,Y)$
be a $M$-ideal in ${\mathcal L}(X,Y)$. For any $A \in {\mathcal L}(X,Y)$ such that $d(A,J)<1$, $\partial_A = \overline {CO} \{x^{\ast\ast} \otimes y^\ast \in \partial_eJ^\ast_1: x^{\ast\ast}(A^\ast(y^\ast))=1,~x^{\ast\ast} \in \partial_e X^{\ast\ast}_1~,y^\ast \in \partial_e Y^\ast_1\}$.
\end{cor}
\begin{rem}
	In particular we have $A^\ast$ attains its norm and hence so are all the adjoints of higher order of $A$. More generally we note that if a $B \in {\mathcal L}(X,Y)$, is such that $B^\ast$ attains its norm, then it attains it at an extreme point $y^\ast \in Y^\ast_1$ (see \cite{Lima}). Since $\|B\|= \|B^\ast(y^\ast)\|$, let $x^{\ast\ast} \in \partial_e X^{\ast\ast}_1$ be such that $x^{\ast\ast}(B^\ast(y^\ast))= (x^{\ast\ast} \otimes y^\ast)(B) = \|B\|$. Now if the $J$ above contains all finite rank operators, $x^{\ast\ast} \otimes y^\ast \in \partial_e J^\ast_1$. It may be worth recalling here that the set of operators whose adjoint attains its norm is dense in ${\mathcal L}(X,Y)$, see \cite{Z}. When $X=Y= \ell^p$ for $1<p<\infty$, $J ={\mathcal K}(\ell^p)$, we have for $A \in {\mathcal L}(\ell^p)$ with $d(A,{\mathcal K}(\ell^p))<\|A\|$, $\partial_A = \overline {CO} \{x \otimes y: A(x)(y)= \|A\|,~x \in \ell^p_1 ~,y \in \ell^q_1\}$, thus extending Theorem 1.4 from \cite{S}.
\end{rem}

Let $Z \subset Y$ be a closed subspace. The following proposition identifies subdifferential limits of vectors in $X/Z$. In what follows, by $\pi: X \rightarrow X/Z$ we denote the quotient map. 
\begin{prop}
Let $Z \subset Y \subset X$ be Banach spaces. Suppose $Y$ is a $M$-ideal in $X$ and $x \in X$ is such that $d(x,Y)<d(x,Z)$. Then for $x'\in X$,
$$lim_{t \rightarrow 0^+} \frac{\|\pi(x)+t\pi(x')\|-\|\pi(x)\|}{t}$$
$$= sup\{re(f(x')): f \in \partial_e (Y^\ast \cap Z^\bot)_1~,~f(x)=1\}.$$
\end{prop}
\begin{proof} By Proposition I.1.7 in \cite{HWW}, we have that  $Y/Z$ is a $M$-ideal in $X/Z$. Hence $Z^\bot = Y^\bot \bigoplus_1 (Y^\ast \cap Z^\bot)$. For any $y \in Y$, $\|\pi(x)-\pi(y)\|=\|\pi(x-y)\| \leq \|x-y\|$.
Thus $d(\pi(x),Y/Z) \leq d(x,Y)<d(x,Z)=\|\pi(x)\|$. Now identifying $(Y/Z)^\ast = Y^\ast \cap Z^\bot \subset (X/Z)^\ast$, the conclusion follows from Theorem 1.	
\end{proof}

We next consider the question, suppose $J$ is a $M$-ideal in $X$ and the formula for $\partial_x$ is valid, when is $d(x,J)<1?$. The following example illustrates how it can depend on the choice of the $M$-ideal.
\vskip 1em
\begin{ex}
	Consider $c_0 \subset \ell^\infty$, a well known example of a $M$-ideal (algebraic ideal) of $\ell^\infty$. Let $x \in \ell^\infty$ be a unit vector, with all but finitely many coordinates coming from the unit circle $\Gamma$, the rest from the open disc. Using the identification of  $\ell^\infty = C(\beta({\bf N}))$, we see that $\partial_e(\partial_x) \subset \Gamma \{\delta(\tau): \tau\in \beta({\bf N})\}$. Thus we get $\partial_x = \overline {CO}\{\overline{t}\delta(k): k\in {\bf N}~, t=x(k)~,|x(k)|=1\}$, where $\delta(\tau)$ is the Dirac measure. Clearly $\delta(k) \in \partial_e (\ell^1 = c_0^\ast)_1$. Now for the ideal
	$J = \{\alpha \in \ell^\infty: \alpha(k) = 0~for~k~with ~|x(k)|<1 \})$, $J$ determines $\partial_x$, as well as $c_0$. However $d(x, c_0) = 1$ and $d(x,J)<1$.
\end{ex}
We give two positive answers depending on the weak$^\ast$-closure of $\partial_e J^\ast_1$ in $J^\ast$ and $X^\ast$.
\vskip 1em

\begin{thm}

Suppose for $X,Y,J$ as in Theorem 1, Suppose $\partial_x = \overline{CO}\{f \in \partial_e J^\ast_1: f(x) = 1\}$.
\begin{enumerate}
\item Suppose $0 \notin \overline{\partial_e J^\ast_1}$, closure in the weak$^\ast$-topology of $J^\ast$.	
\item Suppose
$\overline{\partial_e J^\ast_1} \subset [0,1]\partial_e J^\ast_1$, closure taken on the weak$^\ast$-topology of $X^\ast$. 
\end{enumerate}
 then $d(x,J) <1$.
\end{thm}
\begin{proof}
To see this assuming  1), if $d(x,J) = 1$, choose $\tau \in J^\bot_1 \cap \partial_e X^\ast$ such that $\tau(x) = 1$ . Since $\tau \in \partial_x$ by Milman's converse of the Krein-Milman theorem, $\tau \in \overline{\{f \in \partial_e J^\ast_1: f(x) = 1\}}$, closure in the weak$^\ast$-topology of $X^\ast$. Let $\{f_{\alpha}\}$ be a net from the set in RHS such that $f_{\alpha} \rightarrow \tau$ in the weak$^\ast$-topology of $X^\ast$. Since $\tau(J) = 0$, we get a contradiction. So $d(x,J) <1$.
\vskip 1em
Now suppose we have 2). Since $\tau \in \overline{\partial_e J^\ast_1}$, for the net $\{f_{\alpha}\}$ as in the proof of 1), we get by hypothesis in 2), $\tau = \lambda h$ for $h \in \partial_e J^\ast_1$, $\lambda \in [0,1]$. Since $\|\tau\|=1=\|h\|$, we get $\lambda =1$. Hence  our claim follows as before.
\end{proof}
\vskip 1em

\begin{rem}
	Examples of spaces that satisfy above extremal conditions include spaces  of the form $\{f \in C(\Omega): f(t_{\alpha})=\lambda_{\alpha} f(s_{\alpha})~for~all~\alpha \in \Delta\}$. Here $\Omega$ is a compact Hausdorff space, $\Delta$ is any index set, $$\{t_{\alpha}\}_{\alpha \in \Delta}~,~\{s_{\alpha}\}_{\alpha \in \Delta} \subset \Omega~ and ~\{\lambda_{\alpha}\}_{\alpha \in \Delta} \subset [-1,1].$$
\end{rem}
\vskip 2em
We next consider conditions specific to spaces of compact operators. We recall that in the case of $\ell^p$-spaces, $0$ is a weak-accumulation point of the unit sphere.
\vskip 2em
\begin{prop}
Suppose  $J\subset {\mathcal K}(X,Y) \subset {\mathcal L}(X,Y)$ is a $M$-ideal in ${\mathcal L}(X,Y)$ and it  contains all finite rank operators. Assume further, $0$ is not a weak$^\ast$ extreme point of $\partial_e X^{\ast\ast}_1$ and $\partial_e Y^\ast_1$. Let $A \in {\mathcal L}(X,Y),~\|A\|=1$ and
$\partial_A = \overline {CO} \{x^{\ast\ast} \otimes y^\ast : x^{\ast\ast}(A^\ast(y^\ast))=1~,~x^{\ast\ast} \in \partial_e X^{\ast\ast}_1,~y^\ast \in \partial_eY^\ast_1\}$. Then $d(A,J) < 1$.
\end{prop}
\vskip 1em
\begin{proof}
We will show that $0 \notin \overline{\partial_e J^\ast_1}$. If not, in view of the description of the extreme points, we get two nets indexed by the same set, $\{x^{\ast\ast}_{\alpha}\} \subset \partial_e X^{\ast\ast}_1$ and $\{y^\ast_{\alpha}\} \subset \partial_e Y^\ast_1$ such that the functionals $x^{\ast\ast}_{\alpha} \otimes y^\ast_{\alpha} \rightarrow 0$  in the weak$^\ast$-topology of $J^\ast$. Using weak$^\ast$-compactness, by going through appropriate subnets, if necessary, we may and do assume that $x^{\ast\ast}_{\alpha} \rightarrow x^{\ast\ast}$ and $y^\ast_{\alpha} \rightarrow y^\ast$ in the weak$^\ast$-topology, for some $x^{\ast\ast} \neq 0 \neq y^\ast$. For any $T \in J$, since $T$ is compact, $T^\ast(y^\ast_{\alpha}) \rightarrow T^\ast(y^\ast)$ in the norm. Consequently $x^{\ast\ast}_{\alpha}(T^\ast(y^\ast_{\alpha})) \rightarrow x^{\ast\ast}(T^\ast(y^\ast))$. Therefore $x^{\ast\ast}_{\alpha} \otimes y^\ast_{\alpha} \rightarrow x^{\ast\ast}\otimes y^\ast$ in the weak$^\ast$-topology of $J^\ast$. So $x^{\ast\ast} \otimes y^\ast = 0$ on $J$ . A contradiction since $J$ contains all operators of rank one.
\end{proof}
\vskip 2em

We next apply this to the situation when $X$ is a $M$-ideal in the canonical embedding, in its bidual $X^{\ast\ast}$, the so called $M$-embedded spaces. See Chapter III of \cite{HWW}, for examples among spaces of operators, particularly when for reflexive Banach spaces $X,Y$ ${\mathcal L}(X,Y)$ is the canonical bidual of ${\mathcal K}(X,Y)$ (like in the case of Hilbert spaces and $\ell^p$-spaces).
\vskip 1em
\begin{prop}

Suppose $X$ is a $M$-ideal in $X^{\ast\ast}$. Let $\tau \in X^{\ast\ast}$ be such that $d(\tau,X) < \|\tau\|$, then $\tau$ attains its norm. The set $\partial \tau$ has the same description in all higher odd ordered duals of $X$.
\end{prop}
\vskip 1em
\begin{proof}
Since $X$ is a $M$-ideal in $X^{\ast\ast}$, the canonical projection $\Lambda \rightarrow \Lambda|X$ is a $L$-projection, thus $X^{\ast\ast\ast} = X^{\ast} \bigoplus_1 X^\bot$ for the canonical embedding of $X^\ast$ in $X^{\ast\ast\ast}$. We see that the hypothesis implies there is a $g \in \partial_e X^\ast_1$ such that $\tau(g) =1$. The main result of \cite{R} says that under the hypothesis, $X$ continues to be a $M$-ideal of all higher even ordered duals of $X$. Hence the conclusion follows.
\end{proof}
Let $X$ be canonically embedded in its bidual, $X^{\ast\ast}$. Another interesting question is to consider when the equality,
$$lim_{t\rightarrow 0^+} \frac{\|x+t\tau\|-\|x\|}{t} = sup \{re(\tau(f)): f \in \partial_x\}$$ holds for any $\tau \in X^{\ast\ast}$.
\vskip 1em

Our next proposition answers this question using the notion of the preduality mapping on $X^{\ast\ast}$ (see \cite{GI}). For $x \in X$ considered as an element of $X^{\ast\ast}$, we let $\partial'_x = \{\tau \in X^{\ast\ast\ast}_1: \tau(x) = \|\tau\|\}$. 
\vskip 1em

 We recall that for
a norm attaining functional $\Lambda \in X^{\ast\ast}$, attaining its norm at $f_0 \in X^{\ast}_1$ is a point of norm-weak upper semi-continuity for the preduality map, if for any weak neighbourhood $V$ of $0$ in $X^\ast$, there is a $\delta >0$ such that for any $\Lambda' \in X^{\ast\ast}$, $\|\Lambda'-\Lambda\|< \delta$, $\{g \in X^\ast_1:\Lambda'(g) = \|\Lambda'\|\} \subset \{h \in X^\ast_1: \Lambda(h) = \|\Lambda\|\}+V$.
Part of our next proposition is perhaps folklore, as we are unable to find a reference, we include its easy proof. It is well known that any unitary in a $C^\ast$- algebra satisfies the condition assumed in the following proposition.
\begin{prop}
Let $x \in X$ be a point of norm-weak upper semi-continuity for the preduality mapping of $X^{\ast\ast}$. 	Then $$lim_{t\rightarrow 0^+} \frac{\|x+t\tau\|-\|x\|}{t} = sup \{re(\tau(f)): f \in \partial_x\}$$ for any $\tau \in X^{\ast\ast}$. 
\end{prop}
\begin{proof}
It follows from Lemma 2.2 in \cite{GI} that the hypothesis implies $\partial_x$ is weak$^\ast$-dense (in the weak$^\ast$-topology of $X^{\ast\ast\ast}$) in $\partial'_x$. Since $\tau$ is weak$^\ast$-continuous on $X^{\ast\ast\ast}$, the conclusion follows.

\end{proof}
\begin{rem}
	We note that if $J \subset X$ is a $M$-ideal, then $J$ need not be a $M$-ideal in $X^{\ast\ast}$ ( it is easy to see that in $C([0,1])$, the ideal $J = \{f \in C([0,1]): f([0,\frac{1}{2}])=0\}$ is not an ideal in $C([0,1])^{\ast\ast}$. 
\end{rem}
\section{Application}
As an application of the previous set of ideas, we formulate the compact operator version of the operator norm subdifferential problem and provide a partial solution.
\vskip 1em
Let $X$ be a reflexive Banach spaces and suppose $Y$ has the metric approximation property. Suppose $J \subset Y$ be a $M$-ideal and let $T \in {\mathcal K}(X,Y)$ be such that $\partial_T = \overline {CO} \{x\otimes y^\ast: y^\ast(T(x))= (x \otimes y^\ast)(T)=\|T\|,~~x \in \partial_e X_1,~y^\ast \in \partial_e J^\ast_1\}$. An interesting question that arises from the previous section is that, when can one get a $x_0 \in \partial_e X_1$ such that $\|T\|=\|T(x_0)\|$ and $\partial_{T(x_0)}=\overline{CO}\{y^\ast\in \partial_e J^\ast_1: y^\ast(T(x_0))= \|T(x_0)\|\}$.
\vskip 1em
Since $Y$ and hence $J$ have the metric approximation property, we have that ${\mathcal K}(X,J)$ is a $M$-ideal in ${\mathcal K}(X,Y)$ (see \cite{HWW} Proposition VI.3.1).  Thus from the structure of extreme points of the dual unit ball of spaces of compact operators, we see that the description of $\partial_T$, as before, is in terms of the extreme points of the dual unit ball of the $M$-ideal ${\mathcal K}(X,J)$. 
\vskip 1em
\vskip 1em
	In what follows we will be using a minimax formula for compact operators. This author has recently proved (\cite{R1}), Theorem 1) the validity of this when $J$ is a $L^1$-predual space (i.e., $J^\ast$ is isometric to a $L^1(\mu)$-space, for example, the spaces considered in Remark 8).

\begin{thm}
	Let $X,Y,J$ be as above. Assume that $X$ is a separable space, $0$ is not a weak accumulation point of $\partial_e X_1$ and $0$ is not a weak$^\ast$-accumulation point of $\partial_e Y^\ast_1$ and suppose the compact operator $T$ satisfies the minimax formula, $d(T, {\mathcal K}(X,J))= sup\{d(T(x),J): x \in X_1\}$. If $\partial_T = \overline {CO}\{x\otimes y^\ast: (x \otimes y^\ast)(T) = \|T\|,~x \in \partial_e X_1~,~y^\ast\in \partial_e J^\ast_1\}$, then there is a $x_0 \in \partial_e X_1$ such that, $\|T\|=\|T(x_0)\|$, $\partial_{T(x_0)} = \overline{CO}\{y^\ast \in \partial_e Y^\ast_1: y^\ast(T(x_0))=\|T(x_0)\|\}$.
	\end{thm}
\begin{proof}
	Since $X$ is reflexive and $T$ is compact, let $x_0 \in \partial_e X_1$ be such that $\|T\|= \|T(x_0)\|$. Since ${\mathcal K}(X,J)$ is a $M$-ideal in ${\mathcal K}(X,Y)$, by our arguments given in the proof of Proposition 9, and our hypothesis implies, $d(T, {\mathcal K}(X,J))<\|T\|$. Now by the minimax formula, $d(T(x_0),J) < \|T\|=\|T(x_0)\|$. Now since $J$ is a $M$-ideal, we get the description for $\partial_{T(x_0)}$.
\end{proof}
For a compact Hausdorff space $\Omega$, let $C(\Omega,Y)$ be the space of $Y$-valued continuous functions, equipped with the supremum norm. The minimax formula above was inspired by a classical result of Light and Cheney (see \cite{LC}), for a closed subspace $J \subset Y$ and $f \in C(\Omega,Y)$, $$d(f,C(\Omega,J))= sup\{d(f(w),J): w \in \Omega\}.$$
Analogous to the above theorem, if $J$ is a $M$-ideal in $Y$, then $C(\Omega,J)$ is a $M$-ideal in $C(\Omega,Y)$ (see \cite{HWW}). The following corollary, which assumes these conditions,  is now easy to prove. For $\omega \in \Omega$, $\delta(\omega)$ denotes the Dirac measure at $\omega$.
\begin{cor}
	Suppose $0$ is not a weak$^\ast$-accumulation point of $\partial_e Y^\ast_1$. Let $f \in C(\Omega,Y)$ and suppose, $\partial_f = \overline{CO} \{\delta(\omega)\otimes y^\ast: y^\ast(f(\omega))= \|f\|~,~\omega \in \Omega~,~y^\ast \in \partial_e J^\ast_1\}$. There exists a $\omega_0 \in \Omega$ such that $\|f\|=\|f(\omega_0)\|$, $\partial_{f(\omega_0)}= \overline{CO}\{y^\ast \in \partial_eJ^\ast_1: y^\ast(f(\omega_0))=\|f(\omega_0)\|\}$.
	
\end{cor}
Let $WC(\Omega,Y)$ denote the space of continuous functions when $Y$ is considered with the weak topology, equipped with the supremum norm.
We recall that for $f \in WC(\Omega,Y)$, $\|f\|=Sup_{\omega \in \Omega,y^\ast \in \partial_e Y^\ast_1} |y^\ast(f(\omega))|$. Thus by an application of the separation theorem, $WC(\Omega,Y)^\ast_1 = \overline{CO}\{\delta(\omega)\otimes y^\ast:\omega \in \Omega~,~y^\ast \in \partial_e Y^\ast_1\}$. Now suppose $0$ is not a weak$^\ast$-accumulation point of $\partial_e Y^\ast_1$, then as before we can conclude that for $f \in WC(\Omega,Y)$, if $\Lambda \in \partial_e(\partial_f)$, then $\Lambda=\delta(\omega_0)\otimes y_0^\ast$ for some $\omega_0 \in \Omega$, $y_0^\ast \in \partial_e Y^\ast_1$. In particular, $\|f\|= \|f(\omega_0)\|$. We do not know how to prove a version of Corollary 14 in this set up.

\vskip 1em
The author thanks the referee for his/her remarks and the Journal administration for their efficient handling of this article during the pandemic.

\end{document}